\documentclass[%
aip,
jmp,
amsmath,
amssymb,
reprint,
onecolumn
]{revtex4-2}

\usepackage{graphicx}
\usepackage{bm}

\usepackage[utf8]{inputenc}
\usepackage[T1]{fontenc}
\usepackage{mathptmx}
\usepackage{etoolbox}
\usepackage{amsthm}

\makeatletter
\def\@email#1#2{%
 \endgroup
 \patchcmd{\titleblock@produce}
  {\frontmatter@RRAPformat}
  {\frontmatter@RRAPformat{\produce@RRAP{*#1\href{mailto:#2}{#2}}}\frontmatter@RRAPformat}
  {}{}
}%
\makeatother
\usepackage{enumerate}
\usepackage[multiple]{footmisc}
\usepackage[colorlinks=true, pagebackref=true]{hyperref}
\hypersetup{linkcolor=black, citecolor=black, urlcolor=black}
\usepackage{bbm}  
\usepackage{caption}

\newtheorem{theorem}{Theorem}
\newtheorem{lemma}[theorem]{Lemma}
\newtheorem{proposition}[theorem]{Proposition}
\newtheorem{corollary}[theorem]{Corollary}

\newtheorem{remark}[theorem]{Remark}

\DeclareMathOperator{\law}{law}
\DeclareMathOperator\Var{\text{\rm Var}}
\DeclareMathOperator\Cov{\text{\rm Cov}}
\DeclareMathOperator\cov{\text{\rm Cov}}

\DeclareMathOperator{\rig}{\text{\rm right}}
\DeclareMathOperator{\lef}{\text{\rm left}}
\DeclareMathOperator{\iso}{\text{\rm iso}}
\DeclareMathOperator{\palm}{\text{\rm Palm}}

\newcommand\cB{{\mathcal B}}

\newcommand\cM{{\mathcal M}}
\newcommand\cX{{\mathcal X}}
\newcommand\cY{{\mathcal Y}}

\newcommand\RR{{\mathbb R}}

\newcommand\EE{{\mathbb E}}

\newcommand\uno{{\mathsf 1}} 
\newcommand\one{{\mathsf 1}}

\newcommand\ve{\varepsilon}
\newcommand\ved{{\varepsilon^2}}
\newcommand\vet{{\varepsilon^{-2}}}
\newcommand\veo{{\varepsilon^{-1}}}

\newcommand \etab{B }

\newcommand\hQ{\widehat Q}
\newcommand\hT{\widehat T}
\newcommand\hU{\widehat U}
\newcommand\hS{\widehat S}

\newcommand{\heta}{{\widehat\etab}}
\newcommand{\hB}{{\widehat B}}
\newcommand{\hmu}{{\widehat\mu}}
\newcommand{\tmu}{{\widetilde\mu}}
\newcommand{\hd}{{\widehat \dd}}

\newcommand{\teta}{{\widetilde\etab}}

\newcommand{\tsigma}{{\tilde\sigma}}
\newcommand{\trho}{{\tilde\rho}}

\newcommand{\ta}{{\tilde a}}
\newcommand{\tb}{{\tilde b}}

\newcommand{\tg}{{\tilde g}}
\newcommand{\tr}{{\tilde r}}

\newcommand{\tit}{{\tilde t}}
\newcommand{\tv}{{\tilde v}}

\newcommand{\tx}{{\tilde x}}
\newcommand{\ty}{{\tilde y}}

\renewcommand{\ge}{\geqslant}
\renewcommand{\le}{\leqslant}
\newcommand{\dd}{\mathrm{d}}

\newcommand{\ieta}{{\etab^{\text{\rm iso}}}}

\newcommand{\eqlaw}{\;\overset{\law}{=}\;}
\newcommand{\toas}{\;\overunderset{\text{\rm a.s.}}{\ve\to0}{\longrightarrow}\;}
\newcommand{\tolaw}{\;\overunderset{\law}{\ve\to0}{\longrightarrow}\;}
\newcommand{\toe}{\;\underset{\ve\to0}{\longrightarrow}\;}
\newcommand{\lime}{\underset{\ve\to0}{\lim}}

\newcommand{\sint}{\textstyle{\int}}

  \newcommand{\bab}{{ab}}

\newcommand{\oab}{\overline{ab}}

\newcommand{\oob}{\overline{ob}}

\allowdisplaybreaks

\begin{document}


\title{Multitime fields and hard rod scaling limits}
\author{Pablo A. Ferrari}
\affiliation{Universidad de Buenos Aires \emph{and}  IMAS-UBA-CONICET, Argentina.}  \email{pferrari@dm.uba.ar}
\author{Stefano Olla}%
\affiliation{CEREMADE,
   Universit\'e Paris Dauphine - PSL Research University
\emph{and}  Institut Universitaire de France
\emph{and} GSSI, L'Aquila 
}

\date{November 5, 2025}

\begin{abstract}

A Poisson line process is a random set of straight lines contained in the plane, as the image of the map $(x,v)\mapsto (x+vt)_{t\in\mathbb{R}}$, for each point $(x,v)$ of a Poisson process in the space-velocity plane. By associating a step with each line of the process, a random surface called multitime walk field is obtained. The diffusive rescaling of the surface converges to the multitime Brownian motion, a classical Gaussian field also called Lévy-Chentsov field. A cut of the multitime fields with a perpendicular plane, reveals a one dimensional continuous time random walk and a Brownian motion, respectively.

A hard rod is an interval contained in  $\mathbb{R}$ that travels ballistically until it collides with another hard rod, at which point they interchange positions. By associating each line with the ballistic displacement of a hard rod and associating surface steps with hard rod jumps, we obtain the hydrodynamic limits of the hard rods in the Euler and diffusive scalings. The main tools are law of large numbers and central limit theorems for Poisson processes.

When rod sizes are zero we have an ideal gas dynamics. We describe the relation between ideal gas and hard-rod invariant measures.

\end{abstract}


\maketitle

\section{Introduction}      
To compute the length of a planar curve, Crofton \cite{zbMATH02722727} proposed in 1868 a method involving randomly drawing straight lines within a plane. Let $\ell(\theta,p)$ be the line that has a distance of $p$ from the origin, and whose perpendicular, drawn from the origin of $\RR^2$, makes an angle $\theta$ with the $x$-axis. The measure on the set of lines contained in $\RR^2$ induced by $\ell$ from the Lebesgue measure on the strip $\RR_+\times [0,2\pi)$, is invariant under isometries, see Santaló \cite{MR2162874}. 

In 1945 Lévy \cite{zbMATH02507070,zbMATH03105509,levy-1948} proposes a random surface $\etab:\RR^2\to\RR$ called Brownian motion with several (time) parameters,  a centered Gaussian process with covariances $\Cov(\etab(a),\etab(b)) = \frac12(|a|+|b|-|a-b|)$, where $|\cdot|$ is the Euclidean norm. The one-dimensional trajectory obtained by  cutting the surface with a vertical plane is one dimensional Brownian motion, explaining the nomenclature. 

White noise in $\RR^d$ with control measure $\mu$ is a centered Gaussian process $\omega$ indexed by Borel sets, with covariances $\Cov(\omega(A),\omega(B))= \mu(A\cap B)$. Considering $\mu$ as the Lebesgue measure on the $\theta$-$p$ strip,  Chentsov \cite{chentsov1957levy} describes the Lévy's multitime Brownian motion as the surface $\etab(b) := \omega(ob)$, where $ob$ is the set of points $(\theta,p)$ mapped to lines crossing the segment $\oob$; $o$ is the origin of $\RR^2$. 

In 1975 Maldenbrot \cite{MR388559,zbMATH03516826} considered an isometry invariant marked Poisson line process, and introduced the multitime Poisson field $M(b):= \sum_{\ell\in ob} r(\ell)$, where the sum runs on the set of Poisson lines $\ell$ crossing the segment $\oob$, and $r(\ell)$ is the associated mark; the marks are iid centered random variables with finite variance. He proved that the rescaled function $M$ converges to the multitime Brownian field.   
Ossiander and Pyke \cite{zbMATH03938123} interpret the function $M$ as a set indexed process and shows its convergence to the multitime Brownian motion. Lantuéjoul \cite{Lantuejoul1993,lantuejoul-book} uses $M$ to model and simulate geological phenomena.

Building on Chentsov approach, Ferrari, Franceschini, Grevino and Spohn \cite{ffgs22} introduced a slightly different surface $H_N$, as a function of the empirical measure $N$ of a non necessarily space shift-invariant Poisson line process, and applied it to show hydrodynamic results on the hard rod dynamics. In this paper we review the approach and application, and perform the diffusive rescaling of the fluctuations fields in the non-homogeneous case,  extending results by the authors \cite{fo2024}. 

To better match the relation with hard rods, we map  points in the space-velocity $\RR^2$ to lines contained in the space-time $\RR^2$, by defining $\ell(x,v)=\{(x+vt,t):t\in\RR\}$, the line intersecting the $x$-axis at the point $x$ and having inclination $\alpha(v):=\arctan(1/v)$, the angle whose cotangent is $v$. To incorporate marks, we add a dimension to our space. A point $(x,v,r)$ in the space-velocity-mark  $\RR^3$  represents the line $\ell(x,v)$ with mark~$r$. For $a,b$ in the space-time plane $\RR^2$, denote by $ab$ the set of points $(x,v,r)$ such that the line $\ell(x,v)$ intersects the segment $\oab$. 
Orient the lines in the positive direction of time. Each oriented line $\ell(x,v)$ divides the space-time plane into half-planes denoted $\lef(x,v)$ and $\rig(x,v)$; by convention  the right half plane contains the line. Each line crossing the segment~$\oab$ belongs to one of the sets 
\begin{align}
   ab_+&:= \bigl\{(x,v,r) \in\RR^3: a\in \lef(x,v),\,b\in\rig(x,v)\bigr\}, \label{ab+}\\
   ab_-&:= \bigl\{(x,v,r) \in\RR^3: a\in \rig(x,v),\,b\in\lef(x,v)\bigr\}, \label{ab-}
\end{align}
see Figure \ref{LCS14}.

{\centering
  \includegraphics[width=.8\textwidth]{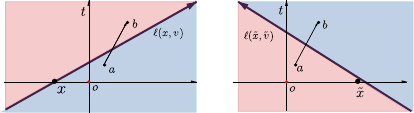}
  \captionof{figure}{The line $\ell(x,v)$ in the left figure belongs to $ab-$ while the line $\ell(\tx,\tv)$ in the right figure belongs to $ab+$. \label{LCS14}}
  \par}

The marked line associated to the point $(x,v,r)$ induces a surface $h_{x,v,r}:\RR^2\to\RR$ which is constant in each half plane determined by $\ell(x,v)$. The height difference between the right and left half planes is $r$, and the height of the half plane containing the origin $o$ is zero. More precisely,
\begin{align}
   \label{h1}
  h_{x,v,r}(b)&:= r\bigl(\one\{(x,v)\in ob_+\}- \one\{(x,v)\in ob_-\}\bigr).
\end{align}
In the left picture of Figure \ref{LCS14} the right half plane of the line $\ell(x,v)$ associated to the point $(x,v,r)$ is at height $0$ and the left one is at height $r$. In the right picture the left plane of $(\tx,\tv,\tr)$ is at height $0$ and the other at heigth $-\tr$.

Let $N:= \sum_{(x,v,r)\in X}\delta_{(x,v,r)}$ be the empirical measure of a finite point configuration $X\subset \RR^3$, and 
define  the function $H_{N}:\RR^2\to\RR$ by
\begin{align}
  H_{N}(b) &:= \sum_{(x,v,r)\in X} h_{x,v,r}(b)= N_1(ob_+)-N_1(ob_-), \label{Hn1}\\
  N_1&:= \sum_{(x,v,r)\in X}r\,\delta_{(x,v,r)}.\label{Hn2}
\end{align}
Define the space of measures
\begin{align}
  \label{cM1}
  \cM&:=\bigl\{\mu\text{ on } \cB(\RR^3): \,\mu(ab)+\mu_2(ab)<\infty,\,\text{ for all }a,b\in\RR^2\bigr\},\\
    d\mu_k(x,v,r)&:= r^k\,d\mu(x,v,r). \label{mk3}
\end{align}
We extend the definition \eqref{Hn1} to every $\mu\in\cM$, by 
 \begin{align}
   \label{Hm1}
   H_\mu(b) := \iiint h_{x,v,r}(b)\, d\mu(x,v,r) =  \mu_1(ob_+)-\mu_1(ob_-).
 \end{align}
Let $\mu\in\cM$ and consider a Poisson process $X$ with intensity measure $\mu$, and empirical measure $N=N[X]$. Denote by $P$ and $E$ the probability and expectation associated to $X$. We have $N\in\cM$ almost surely, and 
 $E H_N=H_\mu$. We refer to $H_N$ as a multi time random walk field, as its one dimensional marginals consist of non homogeneous continuous time random walks, see \eqref{pf1}-\eqref{pf2}.

Fix a rescaling parameter $\ve>0$, later tending to $0$, and consider a Poisson process $X^\ve$ with intensity measure $\ve^{-1}\mu$. The rescaled empirical measure and random walk fields are given by
 \begin{gather}
   \label{ne1}
   N^\ve\varphi := \ve \sum_{(x,v,r)\in X^\ve} \varphi(x,v,r), \qquad N^\ve(A) := N^\ve\uno_A.
    \end{gather}
Notice that $\EE N^\ve_1(ob_\pm)=\mu_1(ob_\pm)$ for all $\ve$, which implies  $E H_{N^\ve}(b) = H_\mu(b)$. 

Define the empirical fluctuation fields 
 \begin{align}
   \label{ee2}
   \etab^{\ve}(b) &:=  \ve^{-1} \bigl(H_{N^{\ve^2}}(b)-H_\mu(b)\bigr),\\
   \heta^{\ve}(b) &:= \ve^{-\frac12} \bigl(H_{N^{\ve^2}}(\ve b)-H_\mu(\ve b)\bigr),
 \end{align}
The processes  $\etab^{\ve}$ and $\heta^{\ve}$ are functions of the same $X^{\ve^2}$. The macroscopic scale is order $1$. The scaling of $\etab^\ve$ is diffusive because the marks associated to lines in $ab_\pm$ have mean $\ve^{-2}\mu_1(ab_\pm)$, and these marks are multiplied by $\ve$, so that the variance is $\mu_2(ab_\pm)$. The same argument applies to $\heta^\ve$, as the masses crossing the segment with extremes $o$ and $\ve b$ have mean $\ve^{-1}\mu_1(ob_\pm)$, so that the variance is order $\ve$. Denote by $d\mu(v,r|z)$ the conditioned law of $\mu$ on the velocity-mark plane $\RR^2$, given that the space coordinate is $z$.
\begin{theorem}[Scaling limits for multitime processes]
   \label{sl1}
Assume $\mu\in\cM$, then,
\begin{align}
     H_{N^\ve}(b) &\toas H_\mu(b), \label{t1a}\\
     (\etab^{\ve}(a),\heta^{\ve}(b))&\tolaw (\etab(a),\heta(b)), \label{t1b}
\end{align}
where $\etab$ and $\heta$ are the Lévy-Chentsov fields associated to the distance $\dd(a,b)= \mu_2(ab)$ and $\hd(a,b)=\hmu_2(ab)$, respectively, where $d\hmu_2(x,v,r) = dx\,r^2 d\mu(v,r|0)$. Furthermore, the processes $\etab$ and $\heta$ are independent. 
\end{theorem}

\paragraph*{Ideal gas and hard rod dynamics} We describe the connection between the multitime walk and Brownian field with the hard rod dynamics, established by FFGS\cite{ffgs22}. 

A point $(x,v,r)$ codifies a length zero particle sitting at $x$ at time zero travelling at velocity $v$, and carrying a mark $r$. The ideal gas dynamics of a configuration $X\in\cX$ is defined by
\begin{align}
  \label{Tt1}
  T_tX := \{(x+vt,v,r):(x,v,t)\in X\},\quad t\in\RR.
\end{align}
There is no interaction between particles.
The trajectory $(T_tX)_{t\in\RR}$ coincides with the marked line configuration associated to $X$, with elements $\ell(x,v)$ carrying the mark $r$. The ideal gas conserves Poisson processes, if $X$ is a Poisson process with intensity $\mu\in\cM$, then $T_tX$ is a Poisson process with intensity measure $\mu T_{-t}$, that is, $\mu T_{-t}\varphi = \iiint \varphi(x-vt,v,r) d\mu(x,v,r)$. In particular, if $\mu$ is space translation invariant, the Poisson with intensity measure $\mu$ is invariant for the ideal gas. 

To a given a point $(y,v,r)$ associate a rod, the interval $(y,y+r)$, carrying a velocity $v$. A hard rod configuration  is a set $Y\in \cX$ satisfying that distinct rods do not intersect. The set of hard rod configurations is denoted
\begin{align}
  \label{cY}
  \cY:=\{Y\subset\RR^3: (y,y+r)\cap (\ty,\ty+\tr)=\emptyset,\;\text{for all distinct }(y,v,r), (\ty,\tv,\tr)\in Y\}
\end{align}
Starting with an $Y\in\cY$, each rod travels deterministically with its velocity, until collision with another, faster or slower, rod.  Just before collision time, the right extreme of the fast rod coincides with the left extreme of the slow one. At collision time the rods swap positions, the left extreme of the updated slow rod goes to the left extreme of the fast rod, and the right extreme of the updated fast rod goes to right extreme of the slow rod. See figure. After collision, each rod continues travelling ballistically with its original velocity. Given a hard rod configuration $Y\in\cY$, the hard rod configuration at time $t$ is denoted $U_tY$, and it is rigorously defined by \eqref{Ut3}. The set $\cY$, is conserved by the dynamics, $Y\in\cY$ if and only if $U_tY\in \cY$.

We discuss the relationship between hard rod and ideal gas invariant measures, in the spirit of the results of FNRW \cite{fnrw} for box-ball systems and their linear decomposition. Theorem~\ref{t14} says that a space-shift invariant and $U$-invariant measure on $\cY$ can related via Palm-theory to a space homogeneous Poisson process if the velocity distribution is absolutely continuous.    

The hard rod evolution can be expressed in terms of the multitime walk field.
\begin{theorem}[Surface representation of the hard rod evolution \cite{ffgs22}]
  \label{hH1}
Define 
\begin{align}
  \label{yx1}
  D_0X &:= \bigl\{(x+H_N(x,0),v,r):(x,v,r)\in X\bigr\}.
\end{align}
Then, $Y=D_0X$ is a hard rod configuration, $Y\in\cY$, and $U_tY$ is given by
\begin{align}
  U_tY&= \bigl\{(x+vt+H_N(x+vt,t),v,r):(x,v,r)\in X\bigr\}. \label{uty}
\end{align}
\end{theorem}
The  configuration $D_0X$ is called the dilation of $X$ with respect to the origin.

In Section \ref{s9} we prove known and new theorems for the length and fluctuation fields of hard rods, by combining Theorem \ref{sl1} with Theorem \ref{hH1}.

  \section{Poisson processes and white noise}
  Let  $\mu$ be a locally integrable intensity measure on $\RR^d$. Let $X$ be a Poisson process on $\RR^d$ with intensity measure $\mu$. Denote the empirical measure associated to $X$ defined on Borel sets $A$ and $\mu$ integrable functions $\varphi$ by
  \begin{align}
    \label{em1}
    N(A):= |X\cap A|,\qquad N\varphi:=\sum _{x\in X} \varphi(x).
  \end{align}
We can look at the Poisson process as a random process indexed by Borel sets:
$(N(A))_{A\in\cB}$, satisfying
\begin{align}
  &\text{$N(A)$ is a Poisson random variable with mean $\mu(A)$}\\
  &\text{$N(A)$ and $N(B)$ are independent if $A\cap B=\emptyset$}
\end{align}
These properties imply that the covariances are given by 
\begin{align}
  \cov\bigl(N(A),N(B)\bigr)&= \mu(A\cap B),\qquad\Var N(A)= \mu(A),\\
 \cov\bigl(N\varphi,N\psi\bigr)&= \mu(\varphi\psi),\qquad\Var N\varphi= \mu\varphi.
\end{align}

\paragraph*{Law of large numbers} Denote $X^\ve$ a Poisson process with intensity measure $\ve^{-1}\mu$ and consider the rescaled empirical measure $N^\ve$ defined on test functions $\varphi$ and Borel sets $A$, by 
\begin{align}
  \label{ne1}
  N^\ve\varphi:=\ve \sum _{x\in X^\ve} \varphi(x), \quad N^{\ve}(A):= N^\ve 1\{\cdot\in A\}.
\end{align}
We have $E N^\ve(A) =\mu(A)$, $E N^\ve\varphi =\mu\varphi$, and the following law of large numbers. 
\begin{lemma}[Law of large numbers for the empirical measure]\label{lnn}
  \begin{align}
 \lime  N^\ve\varphi \toas \mu\varphi,\qquad   \lime  N^\ve(A) \toas \mu(A), \label{ln2}
\end{align}
for $\mu$ integrable functions $\varphi$. 
\end{lemma}
\begin{proof}
  We have $E N^\ve\varphi=\mu\varphi$ and $\Var N^\ve\varphi = \ve (\mu\varphi^2-(\mu\varphi)^2)$. 
By the Superposition Theorem for Poisson processes \cite{kingman}, 
given an  iid family $(X_i)_{i\ge1}$ of Poisson processes with intensity measure $\mu$, for integer $1/\ve$ we have $X^\ve \eqlaw \cup_{i=1}^{\ve^{-1}} X_i$. Then, 
\begin{align}
  \label{ii1}
 N^\ve\varphi = \ve \sum_{i=1}^{1/\ve}  \bigl(\sum_{x\in X_i}\varphi(x)\bigr).
\end{align}
is a normalized sum of iid random variables with finite variance, converging to their common mean $\mu\varphi$. The proofs only uses Chevichev inequality and Borel Cantelli lemma, so it is valid for every realization of the family $X^\ve$.
\end{proof}

\paragraph*{White noise}
Consider locally integrable measure $\mu$ on $\RR^d$. The Gaussian field $\omega$ indexed by Borel sets $A\in\cB$  satisfying
\begin{align}
  &\text{$\omega(A)$ is a centered Gaussian random variable with variance $\mu(A)$},\\
  &\text{$\omega(A)$ and $\omega(B)$ are independent if $A\cap B=\emptyset$},
\end{align}
is called white noise with control measure $\mu$. 
The covariances are given by 
\begin{align}
  \text{$\cov\bigl(\omega(A),\omega(B)\bigr)= \mu(A\cap B)$.}
\end{align}
The covariances characterize the process. 
See for instance Lalley \cite{Lalley2011GAUSSIANPK} for a construction of Gaussian fields. 

\begin{lemma}[Convergence to white noise]
 Let $\mu$ be a locally integrable measure on $\cB(\RR^d)$, and $N^\ve$ the empirical measure of a Poisson process with intensity measure $\ve^{-1}\mu$. Then, for each integrable function $\varphi$, we have
\begin{align}
 \omega^\ve\varphi\, =\, \frac{N^\ve\varphi-\mu\varphi}{\ve^{\frac12}}\; &\tolaw \; \omega\varphi, \label{lf2}
\end{align}
where $\omega$ is white noise with control measure $\mu$.
\end{lemma}
\begin{proof}
Use the central limit theorem applied to the representation \eqref{ii1}. 
\end{proof}


\section{Surfaces generated by marked lines}

A straight line contained in $\RR^2$ can be characterized by two parameters. A popular version is to map the point $(p,\theta)\in\RR_+\times [0,2\pi)$ to the line which is at distance $p$ to the origin and whose perpendicular makes an angle $\theta$ with the $x$ axis. This map is well adapted to produce measures on the space of lines that are invariant by isometries.
We adopt a ballistic codification, seeing the line $\ell(x,v)$ as the trajectory of a traveler moving at constant velocity $v$, visiting the point $x$ at time $0$,
\begin{align}  \notag
  \ell(x,v):=\{(x+vt,t):t\in \RR \},
\end{align}
see Fig.\/\ref{m10}. The function $\ell$ is one to one, the image of $\RR^2$ misses the lines parallel to the space axis.

{\centering
  \includegraphics[width=.40\textwidth, angle=90]{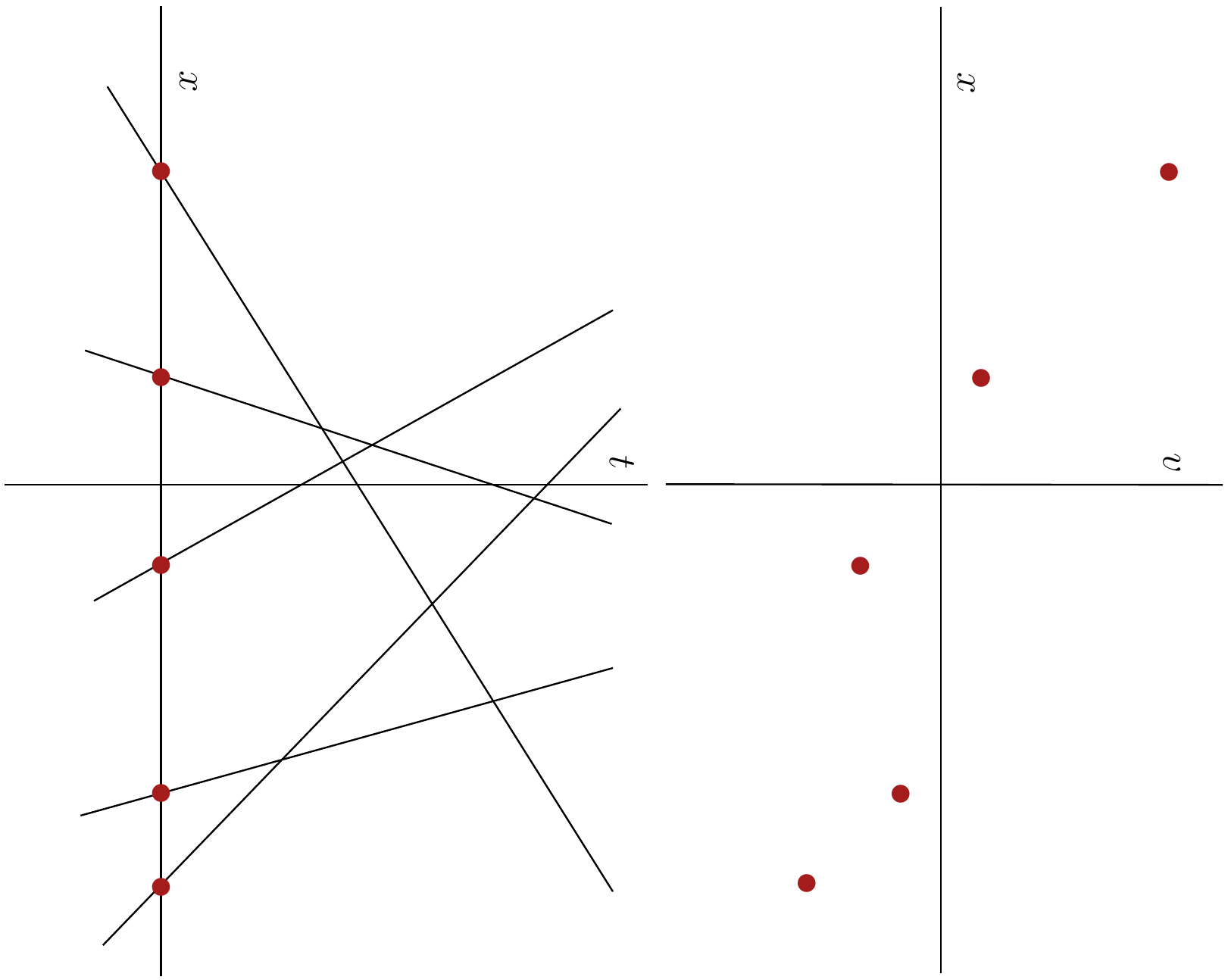}
  \captionof{figure}{Mapping space-velocity points to lines contained in the space-time plane. \label{m10}}
}

Let $a$ and $b$ be distinct  space-time points, denote $\oab$ the segment with extremes $a$ and $b$,
\[\oab:= \bigl\{a u +b (1-u): u\in[0,1]\bigr\}, 
    \quad a,b\in\RR^2, \]
and $ab$ the set of space-velocity points mapped to lines intersecting $\oab$,
\begin{align}
  \label{ab}
  \bab:= \bigl\{(x,v)\in\RR^2: \ell(x,v)\cap \oab\neq\emptyset\bigr\}.
\end{align}
A line $\ell(x,v)$ divides the plane into two half-planes, denoted by
\begin{align}
  \rig(x,v) &:= \bigl\{(\tx,\tit)\in\RR^2: \tx\ge x+\tit v\bigr\},\\
   \lef(x,v) &:= \bigl\{(\tx,\tit)\in\RR^2: \tx< x+\tit v\bigr\}.
\end{align}
The set of lines in $ab$ having $a$ in the left half-plane and $b$ in the right half-plane is denoted $ab_+$; the lines  in $ab\setminus ab_+$ are denoted $ab_-$,
\begin{align}
 \label{ab+}
  ab_+ &:= \{ (x,v)\in ab: a\in\lef(x,v),\, b\in \rig(x,v)\},\\
  ab_- &:= \{ (x,v)\in ab:  a\in\rig(x,v),\, b\in \lef(x,v)\}.
\end{align}

\paragraph*{Surfaces generated by marked lines}

Interpret a point $(x,v,r)\in\RR^3$ as the line $\ell(x,v)$ with an associated mark $r$. Define the height function $h_{(x,v,r)}:\RR^2\to\RR$ by
\begin{align}
  h_{(x,v,r)}(b):= r\,\bigl( 1\{\ell(x,v)\in ob_+\} - \uno\{\ell(x,v)\in ob_-\}\bigr),
       \label{h10}
\end{align}
where $o=(0,0)$ is the origin of the space-time plane. The function $h_{(x,v,r)}$ is constant in the half-planes determined by $\ell(x,v)$, and it has a height difference of~$r$ between the right and left half-planes. The half-plane containing the origin has a height of~$0$. See Figure \ref{LCS-0}.

{\centering
  \includegraphics[width=\textwidth]{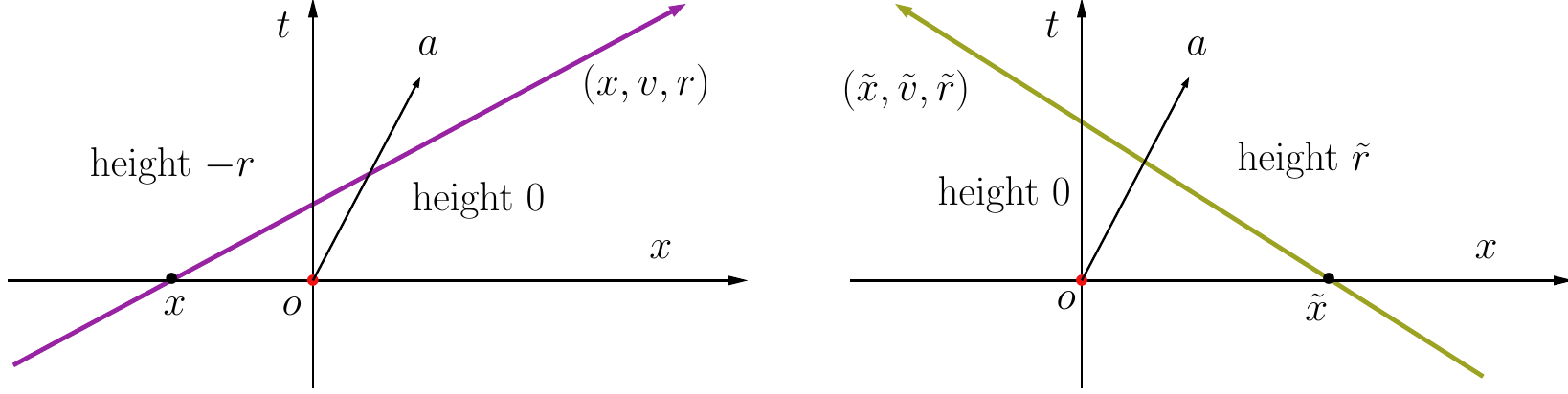}
  \captionof{figure}{Surfaces generated by $(x,v,r)$ and $(\tx,\tv,\tr)$, respectively. We have $h_{(x,v,r)}(a)=-r$ because $o\in\rig(x,v)$ and $a\in\lef(x,v)$, and  $h_{(\tx,\tv,\tr)}(a)=r$ because $o\in\lef(x,v)$ and $a\in\rig(x,v)$.    \label{LCS-0}} 
}\par

Let $N$ be the empirical measure associated to a space-velocity-mark configuration $X$, assume $N\in\cM$ and denote $H_N$ the sum of the simple surfaces,
\begin{align}
 H_N(b) &:= \sum_{(x,v,r)\in X} h_{(x,v,r)}(b)
=    N_1(ob_+)-N_1(ob_-),\label{b70}\end{align}
where $N_1$ was defined in \eqref{Hn2}. 
The surface $H_N$ is null at the origin, $H_N(o)=0$, and the surface differences are given by
\begin{align}
  H_N(b)-  H_N(a) &=   N_1(ab_+)-  N_1(ab_-).
\end{align}

{\centering
  \includegraphics[width=.70\textwidth]{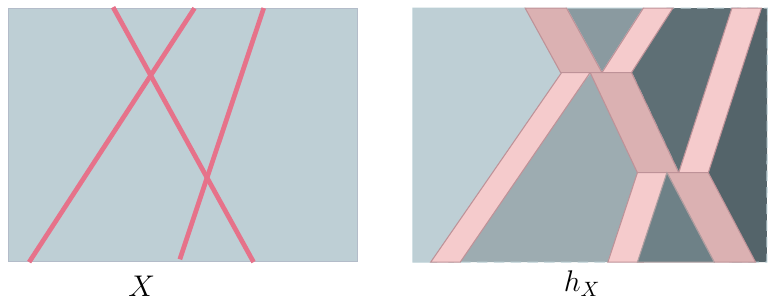}
  \captionof{figure}
  {A point configuration $X$ with 3 lines with marks produces a surface $h_N$ with up to 6 possible heights. The left picture shows 3 lines with marks (invisible), and the right picture is a perspective of the surface, as seen from an observer located at $45^o$ to the left. \label{chr24} }\par
  }

We have mapped a discrete measure $N\in\cM$ to the surface $H_N$, via the quantities $N_1(ab_\pm)$, see Fig.\ref{chr24}. 

\section{Multitime walk field}
\label{srw}

Let $\mu\in \cM$ and $X$ be a Poisson process  on $\RR^3$ with intensity measure $\mu$ and empirical measure~$N$. The random surface  $H_N$, defined by \eqref{b70} is called multitime walk field; its covariances are given by
\begin{align}
  \label{co9}
  \Cov(H_N(a),H_N(b)) = \mu_2(oa\cap ob) = \frac12(\mu_2(oa)+\mu_2(ob)-\mu_2(ab)).
\end{align}
Assume that the measure $\mu\in\cM$ is absolutely continuous in the space coordinate, that is,
\begin{align}
  d\mu(x,v,r) = \rho(x) dx\, d\mu(v,r|x),
\end{align}
with  $\rho$ a locally integrable non-negative function and the family $(d\mu(v,r|x))_{x\in\RR}$ consists of probability measures on the velocity-mark plane $\RR^2$. 
 The height of the field along each straight line $\ell(x,v)$, defined by
\begin{align}
  \label{pf1}
Z_{x,v}(t):=  H_N(x+vt,t),
\end{align}
 is a one-dimensional non-homogeneous continuous time random walk whose generator  $L_{x,v;t}$ at time $t$ acts on test functions $f$ by
\begin{align}
  \label{pf2}
  L_{x,v;t}f(z) &=  \iint \rho(x+(v-w)t)\,d\mu(w,r|x+(v-w)t)\\
  &\qquad\qquad\times\big(\one\{w<v\}f(z+r)+\one\{w>v\}f(z-r)-f(z)\big).
\end{align}
This observation justifies the name ``multitime walk field''.

\paragraph*{Law of large numbers} Let $\mu\in \cM$,  recall the rescaled empirical measure $N^\ve$ defined in \eqref{ne1}, and use \eqref{b70} to define the surface $H_{N^\ve}$. We have the following law of large numbers.
\begin{align}
    H_{N^\ve}(b)
    &\;\toas\; H_\mu(b),\label{l11}
\end{align}
where $H_\mu$ is defined in \eqref{Hm1}. To show the limit \eqref{l11}, denote  $\varphi_A(x,v,r) = r\,\uno \{(x,v,r)\in A\}$, and write $N^\ve_1(A)=  N^\ve\varphi_A$ which by Lemma  \ref{lnn} converges to $\mu\varphi_A=\mu_1(A)$.

\section{Multitime Brownian field}

Let $\mu\in\cM$  and let $\omega_2$ represent white noise in the space-velocity-mark domain $\RR^3$, with control measure $\mu_2$. Define the multitime Brownian field as the process $\etab:\RR^2\to \RR$ given by 
\begin{align}
  \label{e1}
  \etab(a)&:= \omega_2\bigl(oa\bigr),\quad a\in\RR^2.
\end{align}
The process $\etab$ is Gaussian with  covariances 
\begin{align}
  \Cov\bigl(\etab(a),\etab(b)\bigr)
  &=\; \Cov\bigl( \omega_2(oa), \omega_2(ob)\bigr)\\
     &=\; \mu_2(oa\cap ob)\\[1mm]
  &=\;\tfrac12\bigl(\mu_2(oa)+\mu_2(ob)-\mu_2(ab)\bigr),\label{c67}
\end{align}
because $ab = (oa \setminus ob) \cup (ob\setminus oa)$ implying $\mu_2(ab)= \mu_2(oa)+ \mu_2(ob)-2\mu_2(oa\cap ob)$.

An isometry invariant version $\etab^{\iso}$ of the field was introduced by  Lévy \cite{levy-1948}, who called it Brownian motion with several parameters.  The geometrical definition \eqref{e1} was proposed by Chentsov~\cite{chentsov1957levy}, who also observed that $\mu^{\iso}_2(ab)=|b-a|$ is a distance that parametrize the process. For this reason $\ieta$ is  also called Lévy-Chentsov field. The increments of the marginal process along lines $(\ieta(o,a+ut))_{t\in\RR}$ have the same law as those of standard Brownian motion $(W(t))_{t\in\RR}$, for all $a$ and unitary $u\in\RR^2$. Our intensity measure $\mu$ produce fields $\etab$ whose law may be not invariant by isometries, however we still have a (non-homogeneous) multitime Brownian field. The marginals along lines of the field $\etab$ with distance $\dd(a,b)=\mu_2(ab)$ can be written in function of $\ieta$, for each unitary $u\in\RR^2$,
\begin{align}
  \label{ut3}
  (\etab(ut))_{t\in\RR} &\eqlaw \bigl(\ieta\bigl(u\,\dd(o,ut )\bigr)\bigr)_{t\in\RR}.
\end{align}
The marginal profiles along lines satisfy
\begin{align}
\bigl(\etab(t,x+vt)-\etab(0,x)\bigr)_{t\in\RR} 
& \;\eqlaw \;
\bigl(W(\mu_2({a_0a_t})\bigr)_{t\in\RR}\;\;, 
&a_t&:=(t,x+vt) 
;\label{eb1}\\
\bigl(\etab(t,x)-\etab(t,0)\bigr)_{x\in\RR}
& \;\eqlaw \; \
\bigl(W(\mu_2({b_0b_x})\bigr)_{x\in\RR}\;\;,
&b_x&:=(t,x). \label{eb2}
\end{align}
Here $(W(\tau))_{\tau\in\RR}$ is standard two-sided one-dimensional Brownian motion pinned at the origin,  $W(0)=0$. See Takenaka \cite{zbMATH03956132}, McKean \cite{zbMATH03200970}, Fu and Wang \cite{zbMATH07229221}. 

The covariance of height differences is given by
\begin{align}
  \cov\bigl(\etab(b)-\etab(a) ,\etab(\tb )-\etab(\ta)\bigr) &=  \mu_2(ab \cap \ta\tb) \\
  &=\tfrac12\,\bigl(\mu_2(\ta b)+\mu_2(a\tb)-\mu_2(a\ta)-\mu_2(b\tb)\bigr).\label{c68}
\end{align}
To see that, write $ab = ( ab\cap \ta\tb )\cup(  ab\cap a\ta)\cup( ab\cap b\tb )$ and use \eqref{c67} to get
\begin{align}
  \mu_2(ab )&= \mu_2(ab \cap \ta\tb)
                             + \tfrac12\,\bigl(\mu_2(ab)
                             + \mu_2(a\ta)
                             - \mu_2(\ta b)
                             + \mu_2(ab)
                             + \mu_2( b\tb)
                             - \mu_2(a\tb)\bigr),
\end{align}
implying \eqref{c68}.

\section{Surface fluctuations}
\subsection{Euler scaling}
Define
\begin{align}
  \label{e17}
  \etab^\ve(b) := \frac{H_{N^\ve}(b)- H_\mu(b)}{\ve^{\frac12}},\qquad b\in\RR^2.
\end{align}
We have the following convergence
\begin{align}
\etab^\ve\; \tolaw\; \etab,    \label{e10}     
\end{align}
where $\etab$ is the multitime Brownian field with distance $ \dd(a,b):=\mu_2(ab)$.
    \begin{proof}[Proof of \eqref{e10}]
Let $\omega_2$ be white noise in $\RR^3$ with control measure $\mu_2$. By  \eqref{b70},  
\begin{align}
  \etab^\ve(b)&= \frac1{\ve^{\frac12}} \Bigl(N^\ve_1(ob_+)- \mu_1(ob_+)-N^\ve_1(ob_-)+ \mu_1(ob_-)\Bigr) \label{ew1}\\
             &\tolaw \omega_2(ob_+)- \omega_2(ob_-) \eqlaw \omega_2(ob) = \etab(b). \label{ew2}
\end{align}
Denoting $\varphi_A(x,v,r):=r\,\uno\{(x,v,r)\in A\}$,
    \begin{align}
  \frac1{\ve^{\frac12}} (N^\ve_1(A)-\mu_1(A)) =    \frac1{\ve^{\frac12}} (N^\ve\varphi_A-\mu\varphi_A)\; \tolaw\; \omega\varphi_A = \omega_2(A). \label{w2a}
    \end{align}
where the limit follows from \eqref{lf2}. Taking $A\in \{ob,ob_+,ob_-\}$, we get the convergence in \eqref{ew2}; observing that $oa_+$ and $oa_-$ are disjoint, we get identity in law in the same display. 
\end{proof}

Mandelbrot \cite{MR388559} seems to be the first to consider a random walk version of the Lévy's multitime Brownian field. The approach was studied by Mori \cite{zbMATH00034022}, Lifshits \cite{zbMATH03698118}, generalized the fields to indicators. Chapter 8 of Samorodnitsky and Taqqu \cite{zbMATH00614990} for stable isometry invariant fields.
Ossiander \cite{ossiander1984weak} and Ossiander and Pyke \cite{zbMATH03938123} describe this construction for isometry invariant fields and shows its convergence to the Lévy-Chentsov process. Lantuéjoul \cite{Lantuejoul1993} also gives examples of non isometry invariant fields. 
 Fu and Wang study stable fields parametrized by a distance \cite{zbMATH07229221}. Durand and Jaffard \cite{zbMATH06062619} consider the sum of a Poisson Chentsov field to a Levy-Chentsov field.

\begin{remark}[Comparing $H_N$ with Mandelbrot's $M_N$]\rm 
There is a subtle difference between Mandelbrot's function $M_N$ and the field $H_N$ in \eqref{Hn1}. Representing both functions in the space-velocity-mark $\RR^3$, we have 
  \begin{align}
    M_N(b) = N_1(ob) 
             ,\qquad
    H_N(b) = N_1(ob_+)-N_1(ob_-).
  \end{align}
When one crosses a line marked $r$, the sign of the step for $M_N$ depends on the relative position of the origin, while in $H_N$, it depends only on the orientation of the line, so that the surface differences $H_N(b)-H_N(a)$ are covariant by translations. 
\end{remark}

\subsection{Diffusive scaling}
\label{Sff}

  \paragraph*{Time shifts}  We will describe the evolution of the fields as seen from an observer sitting at the level of the surface at $z$, as time $s$ changes. It is convenient to consider a nonnegative density function $\rho$ and an absolutely continuous space locally integrable measure $\mu$ with density~$\rho$. Assume $\mu\in\cM$.
Fix a macroscopic space time point $(z,s)$. Define $\mu_{z,s}$, the space-time translation of $\mu$ by $(z,s)$, the measure acting on test functions $\varphi$ by
\begin{align}
  \label{zs1}
  \mu_{z,s}\varphi := \iiint d\mu(x,v,r)\, \varphi(x +vs-z,v,r).
\end{align}
Let $X^\ve$ be a Poisson process in $\RR^3$ with intensity measure $\ve^{-1}\mu$.  
  Define
  \begin{align}
  H^\ve_{z,s}(x,t) &:= H_{N^\ve}((x,t)+(z,s))-H_{N^\ve}(z,s),\label{hz1}\\
  H_{z,s}(x,t) &:= H_\mu((x,t)+(z,s))-H_\mu(z,s).\label{hz2}
  \end{align}
  These are the empirical and deterministic surfaces as seen from the position $(z,s)$ and height $H_{N^\ve}(z,s)$ and $H_\mu(z,s)$, respectively. Lemma \ref{lnn} implies $H_{N^\ve}(z,s)\toas H_{\mu}(z,s)$, implying that
  \begin{align}
    \label{ln7}
    H^\ve_{z,s}(b) \toas H_{z,s}(b),\quad b\in\RR^2, \quad (z,s)\in\RR^2.
  \end{align}
  
\paragraph*{Fluctuation evolution}
We define fluctuation fields $\heta^{\,\ve}_{z,s}$ and $\teta^{\,\ve}_{z,s}$.

The process $\heta^{\ve}_{z,s}$ looks at the heights deviations at positions  
\begin{align}
  \label{ax}
  \ve (x, t)+(z,s),
\end{align}
relative to the height at $(z,s)$. More precisely, define
\begin{align}
  \label{f33}
  \heta^{\,\ve}_{z,s}(x,t) &:=
  \ve^{-\frac12}\ve^{-1} \bigl( H^\ved_{z,s}(\ve x,\ve t)- H_{z,s}(\ve x,\ve t)\bigr).
\end{align}
The height in \eqref{f33} are of order $\ve$, the factor $\ve^{-1}$ brings them back to 1, and the $\ve^{-\frac12}$ is the right diffusive scaling for $\ve^{-1}$ lines contributing $\ve$ each.

Define
\begin{align}
  \label{ws2}
  d\hmu_{z,s}(x,v,r)&:=  \rho(z-vs)\,dx\,d\mu(v,r|z-vs),\\
 d\hmu_{2,z,s}(x,v,r)&:= r^2\, d\hmu_{z,s}(x,v,r).
\end{align}
Notice that $\hmu_{z,s}$ is space translation invariant.
Denoting  $\heta_{z,s}:\RR^2\to \RR$ the Levy-Chentsov field with distance $\widehat\dd_{z,s}(a,b):=  \hmu_{2,z,s}(ab)$,  we have
   \begin{align}
    \heta^{\,\ve}_{z,s}\tolaw \heta_{z,s}.
   \end{align}

On the other hand, the process $\teta^{\,\ve}_{z,s}$ looks at heights at positions
\begin{align}
  \label{ax}
  (x, t)+(z,s),\quad x,t\in \RR,
\end{align}
from position $(z,s)$, relative to the height at $(z,s)$,
\begin{align}
  \label{f34}
  \teta^{\,\ve}_{z,s}(x,t) :=
  \ve^{-1} \bigl(H^\ved_{z,s}(x,t)-H_{z,s}(x,t)\bigr),
\end{align}
where the $H$ fields are defined in \eqref{hz1} and \eqref{hz2}.

Recalling that $d\mu(x,v,r)= \rho(x)\,dx\,d\mu(v,r|x)$, define
\begin{align}
  \label{ws3}
  d\tmu_{z,s}(x,v,r)&:=  \rho(z+x-vs,v,r)\,dx\,d\mu(v,r|z+x-vs),\\
 d\tmu_{2,z,s}(x,v,r)&:= r^2\, d\tmu_{z,s}(x,v,r).
\end{align}

Denoting  $\teta_{z,s}:\RR^2\to \RR$ the Levy-Chentsov field with distance $\widetilde\dd_{z,s}(a,b):=  \tmu_{2,z,s}(ab)$,  we have
   \begin{align}
    \teta^\ve_{z,s}\tolaw \teta_{z,s}.
   \end{align}
   \begin{theorem}[Joint convergence to multi time Brownian fields]
     \label{bf6}
We have the joint convergence
   \begin{align}
     \label{jo1}
      (\heta^\ve_{z,s}, \teta^\ve_{z,s}) \tolaw (\heta_{z,s}, \teta_{z,s})
   \end{align}
 where $  \heta_{z,s}$ and $ \teta_{z,s}$ are independent.     
   \end{theorem}

 \section{Hard rods}\label{s8}

We start describing the ideal gas. 
Think of an element  $(x,v,r)\in X$ as a massless particle travelling at velocity $v$ and carrying a mark $r$. Given a point configuration $X$ with empirical measure $N\in\cM$, the ideal gas evolution is given by the operator $T_t$ defined by
 \begin{align}
   \label{Tt6}
   T_tX:=\{(x+vt,v,r): (x,v,r)\in X\}.
 \end{align}
 One encounters the particle $(x,v,r)$ at time $t$ at the position $x+vt$, carrying the same velocity and mark.

We define the hard rod dynamics following Boldrighini, Dobrushin and Sukhov \cite{bds}. 
Let $X$ be a point configuration in the space-velocity-mark $\RR^3$ with empirical measure $N=N[X]$. The signed added rod length between $z$ and $x$ is defined by
\begin{align}
  \label{m8}
  m_z^x(N) &:= \sum_{(\tx,\tv,\tr)\in X} r\,\bigl(\one\{z\le \tx<x\}-\one\{x\le \tx<z\} \bigr).
\end{align}
Recall $\cY$ is the space of hard rod configurations defined in \eqref{cY}.
The set of rod configurations with no rod containing $z$ is defined by
\begin{align}
  \label{y0}
  \cY_z:=\{Y\in\cY:z\notin (y,y+r), \text{ for all }(y,v,r)\in Y\}.
\end{align}
Notice that if $z\in Y$, then $Y\in\cY_z$.

The dilation $D_z:\cX\to\cY_z$ of an ideal gas configuration $X$ with respect to $z$ is defined by
 \begin{align}
   \label{dx6}
   D_zX&:= \bigl\{\bigl(x+m_z^x(N),v,r\bigr): (x,v,r)\in X\bigr\}.  
 \end{align}
 Define the contraction $C_z:\cY_z\to\cX$ by
 \begin{align}
   \label{cy7}
   C_zY :=  \bigl\{\bigl(y-m_z^y(N[Y]),v,r\bigr): (y,v,r)\in Y\bigr\}.
 \end{align}
 The length flow along the line $\ell(x,v)$ is defined by
 \begin{align}
   \label{j9}
   j_{[N[X]}(x,v;t):= \sum_{(\tx,\tv,\tr)\in X}\tr \,\bigl(\one\{\tx>x,\tx+t\tv<x+vt\}- \one\{\tx<x,\tx+t\tv>x+vt\}\bigr)
 \end{align}
Notice that $q\in Y$ implies $Y\in\cY_q$. The position at time $t$ of the particle $(q,q+r)$,   with velocity $v$, associated to the point $(q,v,r)\in Y$, is given by
\begin{align}
  \label{test1}
  q+vt + j_{N[C_qY]}(q,v;t).
\end{align}
The position does not depend on $r$. The tagged particle jumps forward by $r'$ when it is crossed by particle with size $r'$ with velocity $w<v$, and backwards by $r'$, if it is crossed by a particle with velocity $w>v$. The net length of particles crossing the tagged rod is just the length flux of ideal gas particles in $C_qY$ along the line $\ell(q,v)$. 

The hard rod evolution is the operator $U_t:\cY\to\cY$, defined by
 \begin{align}
   \label{Ut3}
   U_tY = \bigl\{(q+vt + j_{N[C_qY]}(q,v;t) ,v,r): (q,v,r)\in Y\bigr\}.
 \end{align}
It is useful to express $U_t$ for dilations of ideal gas configurations. 
Let $X\in \cX$ be a configuration with $N[X]\in\cM$ and consider $Y=D_0X\in\cY_0$.  Notice that
 \begin{align}
   \label{yc8}
   \text{if }  y= x+m_0^x(N), \text{ then }  j_{N[C_yY]}(y,v;t) = j_N(x,v;t),
 \end{align}
and the position at time $t$ of the quasi particle at $y=x+m_0^x(N)$ at time 0, is given by
\begin{align}
  \label{y9}
   y_{N[X]}(x,v;t) := x+m_0^x(N[X])+ vt+ j_{N[X]}(x,v;t).
 \end{align}
The evolution of $Y=D_0X\in\cY_0$ expressed in terms of $X$ is given by
 \begin{align}
   \label{Ut4}
  U_tY= U_tD_0X=  \bigl\{(y_{N[X]}(x,v;t),v,r): (x,v,r)\in X\bigr\}.
 \end{align}
The particles evolving with \eqref{Ut3}/\eqref{Ut4} interchange positions when colliding with other particle. See Figure \ref{collision-rule}.
 \begin{center}
  \includegraphics[width=.3\textwidth]{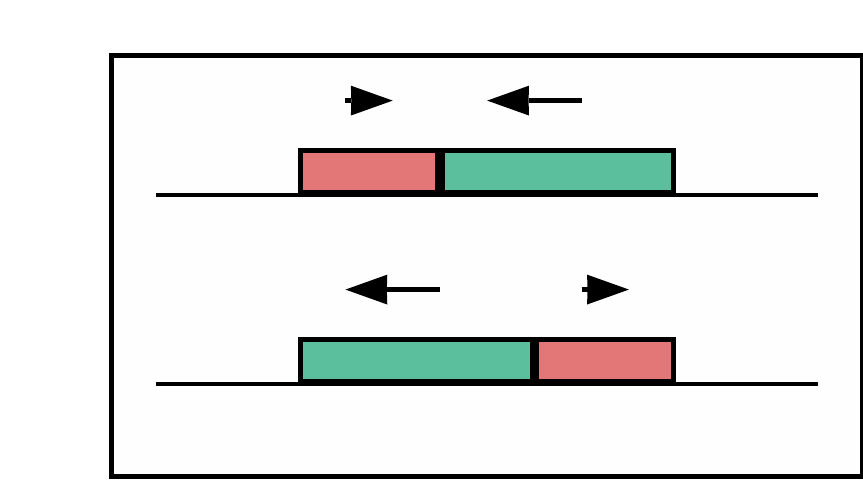}\quad\quad
  \includegraphics[width=.4\textwidth, clip, trim= 8mm 0 0 0]{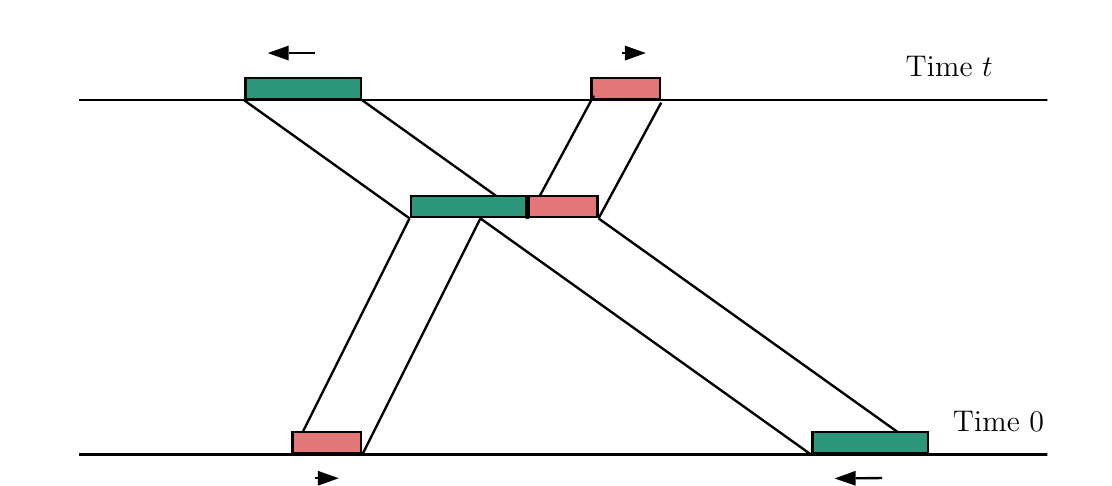}
  \captionof{figure}{Quasi-particle evolution. At collision particles exchange positions. \label{collision-rule} }
\end{center}

 \subsection{Hard rods and multi time walk field}\label{s91}
The hard rod dynamics can be expressed in terms of the multitime walk field, as stated in  Theorem \ref{hH1}. The following lemma follows from the definitions.
\begin{lemma}[Flows and tagged quasiparticles in terms of surfaces]
  \label{lm6}
Assume that a configuration $X$ has empirical measure $N=N[X]\in\cM$, and let $H_N$ be the corresponding surface. Then,
  \begin{align}
       j_N(x,v;t) &= H_N(x+vt,t)-H_N(x,0), \label{jh3} \\
    y_N(x,v;t) &= x+vt+ H_N(x+vt,t), \label{yh3}\\
    m_0^x[N] &= H_N(x,0).\label{mh3}
  \end{align}
\end{lemma}

 \begin{proof}[Proof of Theorem \ref{hH1}]
Use \eqref{mh3} to show that 
$D_0X$ is the configuration $Y$ defined by \eqref{yx1}.
The identity between  \eqref{uty} and \eqref{Ut4} follows from \eqref{yh3}. 
\end{proof}

 \subsection{Hard rod evolution as seen from a tagged particle}

 Consider $X\in \cX$ with empirical measure $N\in\cM$. Let  $Y=D_0X \in\cY_0$ and denote
\begin{align}
     o_t(X ) &:= j_N(0,0;t),\qquad  X\in\cX, \label{hh13}
\end{align}
the position of a zero-length zero-velocity particle
starting from $z=0$. Indeed $o_t(X )= y_N(0,0;t)$, and this position does not depend on $r$. The hard rod  operator $\hU_t$ describing the hard rod dynamics as seen from the tagged zero-length zero-velocity particle is
given by
\begin{align}
  \hU_t Y  :=         D_0T_t C_0Y ,\qquad Y \in \cY_0.\label{hh11}
\end{align}
Let  $S_z$ be the space shift operator defined by
\begin{align}
  \label{shift}
  S_zX  := \{(x+z,v,r):(z,v,r)\in X \}.
\end{align}
Define the {hard rod dynamics} for configurations $Y\in\cY_0$ by
\begin{align}
 U_t Y & := S_{o_t(C_0Y )}\hU_t Y ,\qquad Y \in \cY_0,\label{hh12}
\end{align}
$o_t(C_0Y )$ is the position  at time $t$ of the particle starting at $o:=(0,0,0)$, added to the configuration~$Y \in\cY_0$ at time~$0$. Since the  particle   $o$ has zero length, its addition does not modify the evolution of the rods of~$Y $.

If $Y \in\cY\setminus\cY_0$, then $Y$ has a rod $(q,v,r)$ containing the origin, that is, with $q<0<q+r$. Since $S_{q}Y \in\cY_0$, we can use  \eqref{hh12} to define
\begin{align}
  U_t Y  := S_{-q}U_t S_{q}Y ,\qquad Y \in\cY\setminus\cY_0. \label{Ut5}
\end{align}
The reader can prove the equivalence of the definitions:
\begin{lemma}[Equivalence of definitions]
For any $Y\in\cY$ such that $N[Y]\in\cM$, the definitions of $U_t$ in \eqref{Ut3}-\eqref{Ut4} and \eqref{hh12}-\eqref{Ut5} coincide.
\end{lemma}

\subsection{Invariant measures for the ideal gas}
We say that a measure $P$ on $\cX$ is {$T$-invariant} if $PT_{-t}=P$ for all $t\in\RR$ and that  $P$ is {space shift invariant} or $S$-invariant  if $PS_{-z}=P$ for all $z\in\RR$. The inverse the operator $T_t$ and $S_z$ are $T_{-t}$ and $S_{-z}$, respectively. 

Several works have proven that under mild conditions on the marginal distribution of the velocities, including absolutely continuous velocity distributions, the invariant measures for this process are mixtures of homogeneous Poisson processes, Breiman \cite{zbMATH03190769}, Weiss and Herman \cite{zbMATH03185403}.  Eyink and Spohn \cite{MR1766344} show that the set of ergodic translation invariant measures that have finite density, second moment velocities and finite entropy per volume consist of Poisson processes with constant intensity. Goldstein \cite{zbMATH03562241} shows ergodic properties of the ideal gas. 

The Mapping Theorem \cite{kingman} for Poisson processes implies that if $X$ is a Poisson process with intensity measure $\mu$, then $T_tX$ is a Poisson process with intensity measure $\mu T_{-t}$; in particular, $|T_tX\cap A)|$ is a Poisson random variable with mean $\mu T_{-t}(A)$, for each Borel set $A$ with $\mu(A)<\infty$. The empirical measure of $T_tX$ is $NT_{-t}$.

The first moment measure of a point process $X$ with law $P$ is the expectation of the empirical measure $N$ on $\RR^3$ defined by $\mu:= \int N[X] P(dX)$; $\mu$ is the mean measure.
If $P$ is $S$-invariant and $\mu$ is the mean measure of $P$, then $\mu$  takes the form
\begin{align}
  \label{ftf40}
    \mu(dx,dv,dr) = \rho\,dx \, \nu(dv,dr),
  \end{align}
where the density $\rho$ is a positive constant and $\nu$ is a probability measure on $\RR\times \RR_{\ge0}$.

An $S$-invariant measure $P$ on $\cX$ is $S$-mixing if
  \begin{align}
    \lim_{z\to\pm\infty}\; \bigl|\sint g\,\tg S_{z}\, dP - (\sint g\, dP)\,(\sint \tg \,dP)\bigr| =0,
  \end{align}
  for test functions $g,\tg:\cX\to\RR_{\ge0}$ with finite expectation and compact space support.

Let $X$ be a point process in $\cX$ with distribution $P$ and empirical measure $N\in\cM$, such that the $n$-point correlation function of $P$ is absolutely continuous with density $\rho_n$, $n\ge1$. That is, $\rho_n:(\RR^3)^n\to \RR_{\ge0}$ satisfy
\begin{gather}
  \label{corr1}
\sint\bigl[N(A_1)\dots N(A_n) \bigr] \,P(dX)=\sint_{A_1}\dots\sint_{A_n}
  \rho_n(\zeta_1,\dots,\zeta_n)\,d\zeta_i\dots d\zeta_n,\\
  \text{where $\zeta_i=(x_i,v_i,r_i)$ and $d\zeta_i= dx_i\,dv_i\,dr_i$}, \label{zi7}
\end{gather}
for any collection of bounded, pairwise disjoint Borel sets $A_1,\dots,A_n \subset \RR^3$. The particle density is given by $\rho:=\rho_1$.
 \begin{proposition}[Mixing and Poisson]
   \label{pp5}
Let $X$ be a point process on $\cX$ with distribution $P$, and correlations $\rho_n$ satisfying \eqref{corr1}. If $P$ is $T$-invariant and $S$-mixing,
  then $X$ is a Poisson process with intensity~$\rho=\rho_1$.  
\end{proposition}

\begin{proof}
Since $P$ is $T$-invariant, the one-point correlation function of $P$ must be $S$-invariant: $\rho(x,v,r)=\rho(x+vt,v,r)$.  By $T$-invariance, the $n$-correlation function $\rho_n$ satisfies
  \begin{align}
    \rho_n(\zeta_1,\dots,\zeta_n)& = \rho_n(\zeta_1-v_1t, \dots, \zeta_n-v_nt),\quad t\in\RR,
  \end{align}
  where abusing notation $\zeta_i-tv_i:= (x_i-tv_i,v_i,r_i)$. Since $P$ is $S$-mixing,
  \begin{align}
   \lim_{t\to\infty}  \rho_n(\zeta_1-v_1t, \dots, \zeta_n-v_nt) = \rho(\zeta_1)\cdots \rho(\zeta_n),
  \end{align}
  which is the $n$-correlation function of a Poisson process with intensity $\rho$.
Proposition 4.12 in Last-Penrose \cite{last-penrose} implies that $X$ is a Poisson process with intensity $\rho$.
\end{proof}

\paragraph*{Discrete velocities} 
Let  $V$ be a finite or countable set of velocities and for each $w\in V$, denote  $X_w:= \{(q,v,r)\in X: v=w\}$ the subset of points in $X$ with velocity $w$. Denote by $P_w$ the distribution of $X_w$.
\begin{proposition}[Mixing and independence for discrete velocities]
  \label{mixing-13}
  Let $P$ be an $S$-invariant probability on $\cX$ with marginals $(P_w)_{w\in V}$ and first moment measure $\rho_1\in\cM$. (a) If the $w$-marginals are independent, then $P$ is $T$-invariant. (b) If $P$ is $S$-mixing and $T$-invariant, then the $w$-marginals are independent.
\end{proposition}

\begin{proof}
 (a) Under the free gas dynamics the $w$-marginal at time $t$ is a translation by $wt$ of the time zero marginal, which has law $P_w$ for all $t$. Shifted marginals are also independent, so the superposition of the marginals at time $t$ has law $P$, showing $P$ is $T$-invariant.

 (b) Let $X=\cup_wX_w$ have $T$-invariant and mixing distribution $P$. Then the marginal law of $X_w$ is also mixing. Take measurable sets $A_w\subset \cX_w$, all depending on the same bounded interval of positions. For any finite number of velocities $w$, use $T$-invariance to get
 \begin{align}
   P\bigl(\underset{w}\cap\{X_w\in A_w\} \bigr) = P\bigl(\underset{w}\cap \{X_w\in S_{tw}A_w\}\bigr)\;  \underset{t\to\infty}{\longrightarrow}\; \prod_wP\bigl(\{X_w\in A_w\} \bigr),\notag
 \end{align}
by $S$-invariance and $S$-mixing. 
\end{proof}

\subsection{Palm measures and hard rod invariant measures}

\paragraph*{Palm measures} Recall $\cY_0$ is the set of hard rod configurations with no rod containing the origin, as defined in \eqref{y0}. 
Let $Q$ be an $S$-invariant measure on $\cY$ giving positive mass to $\cY_0$, $Q(\cY_0)>0$. Define the Palm measure of $Q$ by
\begin{align}
 \palm(Q) := Q(\,\cdot\,|\cY_0),
\end{align}
the measure  $Q$ conditioned to the event ``no rod contains the origin''.
For $Y\in\cY_0$ and $z\in\RR$, define
\begin{align}
  b(z,Y) :=  
  \text{ solution $b$ of } \int_0^b 1\{q\notin \cup_{(y,v,r)\in Y}(y,y+r)\} dq = z.
\end{align}
The Lebesgue measure of the empty space in $Y$ between $0$ and $b(z,Y)$ is $z$, if $z$ is positive and $-z$ otherwise. Define the empty space shift by $z$ by
\begin{align}
  \label{101}
\hS_zY :=  S_{b(z,Y)}Y = D_0S_zC_0Y, \quad Y\in \cY_0, \;z\in\RR. 
\end{align}
We say that a measure $\hQ$ on $\cY_0$ is empty space shift invariant, or  $\hS$-invariant, if $\hQ = \hQ\hS_z$ for all $z\in\RR$.
\begin{lemma}[$S$ and $\hS$ invariance]
  \label{102}
  Let $X$ be distributed with $P$ on $\cX$ with first moment $\rho_1\in \cM$. Then, the distribution of $Y:=D_0X \in\cY_0$ is $PC_0$, and
  \begin{align}
    \text{ $P$ is $S$-invariant\quad if and only if \quad $PC_0$ is $\hS$-invariant. }
  \end{align}
\end{lemma}
\begin{proof}
  The lemma follows from \eqref{101}.
\end{proof}

\emph{Inverse Palm}. Given an $\hS$-invariant measure $\hQ$ on $\cY_0$, denote $Q:=\palm^{-1}\hQ$, the inverse-Palm measure defined by
\begin{align}
  \int \varphi(Y) Q(dY) := \int \frac{b(z,Y)}{\int  b(z,Y)\,\hQ(dY)} \frac{1}{b(z,Y)}\int_0^{b(z,Y)} \varphi(S_xY) dx\;\hQ(dY), \quad z\neq 0. 
\end{align}
We multiplied and divided by $b(z,Y)$ to make transparent the following interpretation. 
To obtain a sample of $P$, choose $Y$ with the size biased law $b(z,Y)\hQ(dY)/\int  b(z,Y)\hQ(dY)$ and then shift the origin to a random point uniformly distributed in the interval $(0, b(z,Y))$. The resulting measure $Q$ is $S$-invariant and its definition does not depend on the choice of~$z\neq 0$.

\begin{proposition}[Free gas and hard rod invariance]
  \label{propo25}
  If $P$ is an $S$-invariant measure on $\cX$ with first moment $\rho_1\in \cM$, then $PC_0$ is $\hS$-invariant on $\cY_0$, and 
  \begin{align}
    P \text{ is } T\text{-invariant} \quad\text{if and only if }\quad  PC_0 \text{ is } \hU\text{-invariant}.
  \end{align}
\end{proposition}
The proposition says that $P$ is invariant for the ideal gas dynamics if and only if $PC_0$ is invariant for the hard rod dynamics as seen from the tagged rod $o=(0,0,0)$, defined in \eqref{hh11}. In terms of $\hQ$ and $Q:=\palm^{-1}\hQ$, it says that $\hQ$ is $\hU$-invariant if and only if $QD_0$ is $T$-invariant. 
\begin{proof}
Let $X$ be distributed with $P$ and $Y :=D_0X$. If $P$ is $T$-invariant, then by \eqref{hh11}, 
$  \hU_t Y = D_0T_tX \stackrel{\text{law}}=  D_0X = Y$. Hence $ PC_0$ is $\hat U$-invariant, as $Y$ has law $ PC_0$.
Reciprocally, 
$T_tX
    =  T_t C_0Y
    =  C_0 \hT_t Y
    \stackrel{\text{law}}=  C_0Y = X$.
\end{proof}

\begin{proposition}[Harris \cite{Harris71}]
  \label{harris}
  Let $Q$ be an $S$-invariant measure on $\cY$ and $\hQ$ its Palm measure on $\cY_0$. Then
  \begin{align}
    \text{$Q$ is $U$-invariant\quad if and only if\quad $\hQ$ is $\hU$-invariant}
  \end{align}
\end{proposition}
 See also Port-Stone \cite{PortStone73}, and FNRW \cite{fnrw} for a similar discrete deterministic dynamics.

Wraping up, we get a hard rod version of Theorem 1.5 in FNRW \cite{fnrw} for the box ball system.
\begin{theorem}[Factorization of $S$-mixing invariant measures]\label{t14}
Let $Q$ be an $S$-invariant and $U$-invariant measure on $\cY$. Let $\hQ:=Q(\cdot|\cY_0)$ and assume that $P:=\hQ D_0$ on $\cX$ is $S$-mixing. If the velocity marginal is absolutely continuous, then $P$ is the law of an $S$-invariant Poisson process on $\cX$. If the velocities concentrate on a finite or countable set $V$, then  $P$ has independent mixing marginals $P_w$, $w\in V$. 
\end{theorem}
\begin{proof}
 Since $Q$ is $S$-invariant and $U$-invariant, Theorem \ref{harris} implies $\hQ$ is $\hU$-invariant which implies $\hQ D_0$ is $T$-invariant, by Proposition \ref{propo25}. Since by hypothesis $\hQ D_0$ is mixing, Proposition \ref{pp5} implies $\hQ D_0$ is a Poisson process. Use Proposition \ref{mixing-13} to prove the discrete part.
\end{proof}

 \section{Hard rod hydrodynamics}
\label{s9}

 \subsection{Law of large numbers}
Let $\mu\in\cM$ and  $ X^\ve$ be a Poisson process with intensity measure $\ve^{-1}\mu$. By Lemma \ref{lm6}, the rescaled hard rod configuration is
\begin{align}
  \label{ye1}
   Y^\ve &:= \bigl\{(x+H_{N^\ve}(x,0),v,r):(x,v,r)\in X^\ve\bigr\}.
\end{align}
In this configuration the sum of the absolute value of the length of rods in bounded sets is of order $1$, because the intensity is multiplied by $\ve^{-1}$ and the rod lengths are multiplied by $\ve$. Using \eqref{Ut4} and \eqref{jh3}, the hard rod  empirical measure in the Euler scaling is given by
 \begin{align}
    K^\ve_t\varphi &:= \ve\sum_{(y,v,r)\in U_t Y^\ve}\,r\,\varphi(y,v,r)
                   = \ve\sum_{(x,v,r)\in X^\ve}\,r\,\varphi\bigl(y_{N^\ve}(x,v;t),v,r\bigr),\label{kt3}\\
    y_\nu(x,v;t)   &:= x+ vt +H_\nu(x+vt,t).\label{jf0}
 \end{align}
The point $y_{N^\ve}(x,v;t)$ is the position at time $t$ of the quasi particle or tagged rod, started at the dilated position $y_{N^\ve}(x,v;0)= x+H_{N^\ve}(x,0)$.
Its expectation is given by 
\begin{align}
  \label{ey}
   y_\mu(x,v;t) &= E y_{N^\ve}(x,v;t).
\end{align}

\begin{theorem}[Law of large numbers for hard rods]
The law of large numbers holds for the tagged quasi particle, the mass at time $t$ and the hard rod empirical measure, as follows 
\begin{gather}
   y_{N^{\ve}}(x,v;t) \;\toas\; y_\mu(x,v;t), \label{lyt}\\
 \ve m_0^z(N^\ve_t)\;\toas\; \iiint \one\{x+vt\in[0,z]\}\,r\, d\mu(x,v,r)=:m_0^z(\mu_t ),\label{lmz}\\
   K^\ve_t\varphi  
  \; \toas\; \iiint d\mu(x,v,r)\,r\,\varphi(y_\mu(x,v;t),v,r)=:K_t\varphi.\label{kt4}
 \end{gather}

\end{theorem}
Consider $\mu\in\cM$ absolutely continuous in the first coordinate, with a continuous density $\rho$ and conditional probabilities $(\mu(v,r|x))_{x\in\RR}$,
\begin{align}
    \label{mac}
    d\mu(x,v,r) = \rho(x)\,dx\,d\mu(v,r|x).
\end{align}
The density and conditional probabilities of $\mu_t$ are denoted by
\begin{gather}
d\mu_t(v,r|x):=  d\mu(v,r|x-vt), \qquad    \rho_t(x):= \iint \rho(x-vt) \,  d\mu(v,r|x-vt),\label{rot}\\
 \qquad d\mu_t(x,v,r) = \rho_t(x)dx\,d\mu_t(v,r|x). \label{mvr}
\end{gather}
Denote by $\sigma(x,t)$ the space partial derivative of the function $H_\mu$ at the space-time point $(x,t)$,
\begin{align}
  \sigma(x,t)&:= \partial_x H_\mu(x,t)= \rho_t(x)\,\iint r \,   d\mu_t(w,r|x).   \label{sxt}
\end{align}

Assume now for simplicity that $d\mu(x,v,r) = \rho(x,v,r) dx  dv dr$.
Then we have for the density at time t: $ \rho_t(x,v,r) =  \rho(x-vt,v,r)$,
and $ \sigma(x,t) = \iint r \rho_t(x,v,r)  dv dr$.

Define $ \mathcal Z(x,t) = x+H_\mu(x,t) = y_\mu(x,0,t)$.

The corresponding (macroscopic) density for the hard rods is given by
\begin{align}
  \trho_t(y_\mu(x-vt,v;t),v,r)=  \trho_t(\mathcal Z(x,t) , v,r)=
  \frac{\rho_t(x,v,r)}{1+\sigma(x,t)}
  =  \frac{\rho(x-vt,v,r)}{1+\sigma(x,t)} ,
  \label{eq:trho}
\end{align}
with
\begin{align}
  \label{eq:K}
  K_t\varphi = \iiint r\,  \varphi(q,v,r) \, \trho_t(q,v,r)\,dydvdr.
\end{align}
We have also the inverse relation
\begin{align}
   \rho_t(x,v,r)= \frac{\trho_t(\mathcal Z(x,t),v,r)}{1 - \tilde\sigma( \mathcal Z(x,t),t)},
  \label{eq:trhoinv}
\end{align}
where
\begin{align}
  \label{eq:2}
  \tilde\sigma(q,t) =  \iint r \trho_t(q,v,r)\  dv dr =
  \frac{\sigma(\mathcal Z^{-1}(q,t),t)}{1+ \sigma(\mathcal Z^{-1}(q,t),t)}.
\end{align}

Then from \eqref{kt4} and the change of variable
$x\to q = \mathcal Z(x,t) = x+H_\mu(x,t)$ we have
\begin{align}
  \label{kt7}
  K_t\varphi   &= \iiint r\,  \varphi\bigl(y_\mu(x,v;t),v,r\bigr) \rho(x,v,r) dxdvdr\\
               &= \iiint r\,  \varphi\bigl(x+vt+H_\mu(x+vt,t),v,r\bigr) \rho(x,v,r) dxdvdr\\
  &= \iiint r\,  \varphi\bigl( \mathcal Z(x,t),v,r\bigr) \rho_t(x,v,r) dxdvdr\\
               &= \iiint r\,  \varphi\bigl(q,v,r\bigr)
                 \rho_t(\mathcal Z^{-1}(q,t),v,r)\partial_q \mathcal Z^{-1}(q,t) dqdvdr\\
\end{align}
i.e.
\begin{align}
  \label{eq:25}
  \trho_t(q,v,r) =  \rho_t(\mathcal Z^{-1}(q,t),v,r)\partial_q \mathcal Z^{-1}(q,t),
\end{align}
where
\begin{align}
 \partial_q\mathcal Z^{-1}(q,t) = 1- \tsigma(q,t)\label{eq:4}
\end{align}
  Then by \eqref{eq:K} and \eqref{kt7} we have the identification
  \begin{align}
    \label{eq:id}
    \trho_t(q,v,r) =  \rho_t(\mathcal Z^{-1}(q,t),v,r) \left( 1- \tsigma(q,t) \right)
    = \frac{\rho_t(\mathcal Z^{-1}(q,t),v,r)}{1+ \sigma(\mathcal Z^{-1}(q,t),t)}.
  \end{align}

  We want to identify $V^{\text{eff}}(q,v,t)$ such that
  \begin{align}
    \label{eq:3}
    \partial_t \trho_t(q,v,r) + \partial_q\left(V^{\text{eff}} (q,v,t) \trho_t(q,v,r)\right) = 0
  \end{align}
  Then by \eqref{eq:2} we have
  \begin{align}
    \label{eq:1}
    \partial_t \tilde\sigma(q,t) =  - \partial_q \iint r V^{\text{eff}} (q,v,t) \trho_t(q,v,r)\  dv dr.
  \end{align}
  By \eqref{eq:4}
  \begin{align}
    \label{eq:5}
    \partial_q \partial_t \mathcal Z^{-1}(q,t) =
    \partial_q \iint r V^{\text{eff}} (q,v,t) \trho_t(q,v,r)\  dv dr.
  \end{align}
  Taking time derivative in \eqref{eq:id} we have
  \begin{align}
    \label{eq:6}
     \partial_t \trho_t(q,v,r) &= (\partial_x \rho_t)(\mathcal Z^{-1}(q,t),v,r)
      \left(v + \partial_t \mathcal Z^{-1}(q,t)\right)  \left( 1- \tsigma(q,t) \right)\\
      &\qquad- \rho_t(\mathcal Z^{-1}(q,t),v,r) \partial_t \tsigma(q,t) \\
      &=  (\partial_x \rho_t)(\mathcal Z^{-1}(q,t),v,r)
      \left(v + \partial_t \mathcal Z^{-1}(q,t)\right)  \partial_q\mathcal Z^{-1}(q,t)\\
      &\qquad+ \rho_t(\mathcal Z^{-1}(q,t),v,r)  \partial_q \partial_t \mathcal Z^{-1}(q,t)\\
      &=\left( \partial_q \rho_t (\mathcal Z^{-1}(q,t),v,r)\right)
      \left(v + \partial_t \mathcal Z^{-1}(q,t)\right) \\
      &\qquad+ \rho_t(\mathcal Z^{-1}(q,t),v,r)  \partial_q \partial_t \mathcal Z^{-1}(q,t)\\
      &= \partial_q\left(  \rho_t (\mathcal Z^{-1}(q,t),v,r)
        \left(v + \partial_t \mathcal Z^{-1}(q,t)\right)    \right)\\
      &= \partial_q\left( \frac{ \trho_t (q,v,r)}{1-\tilde\sigma(q,t)}
        \left(v + \partial_t \mathcal Z^{-1}(q,t)\right)    \right)
   \end{align}
  that identify
  \begin{align}
    \label{eq:7}
    V^{\text{eff}}(q,v,t)
    &= \frac{v + \partial_t \mathcal Z^{-1}(q,t)}{1-\tilde\sigma(q,t)}= v +  \frac{v \tilde\sigma(q,t) + \partial_t \mathcal Z^{-1}(q,t)}{1-\tilde\sigma(q,t)}\\
    &= v +  \frac{v \tilde\sigma(q,t) - \tilde\pi(q,t)}{1-\tilde\sigma(q,t)},\label{tpi4}
  \end{align}
where
  \begin{align}
    \label{eq:8}
 \tilde\pi(q,t):= \iint r w \trho_t(q,w,r) dw dr.
  \end{align}
  If $q= \mathcal Z(x,t)= x+H_\mu(x,t)$, then
  $\mathcal Z^{-1}(q,t) = q- H_\mu(\mathcal Z^{-1}(q,t),t)$,
  that implies
  \begin{align}
    \label{eq:9}
    \partial_t \mathcal Z^{-1}(q,t) &= -
    \frac{(\partial_t H_\mu)(\mathcal Z^{-1}(q,t),t)}{1+(\partial_x H_\mu)(\mathcal Z^{-1}(q,t),t)}
    =   \frac{(\partial_t H_\mu)(\mathcal Z^{-1}(q,t),t)}{1+\sigma(\mathcal Z^{-1}(q,t),t)}\\
    &=  \frac{\iint r w \rho_t(\mathcal Z^{-1}(q,t),w,r) dw dr}{1+\sigma(\mathcal Z^{-1}(q,t),t)}
    = - \iint r w \trho_t(q,w,r) dw dr,
\end{align}
  and \eqref{tpi4} follows.

  As consequence  $\trho_t(q,w,r)$ is solution of the align
  \begin{align}
    \label{eq:GHD}
    \partial_t \trho_t(q,v,r) + \partial_q\left(
      \left[v+ \frac{\iint r' (v - w) \trho_t(q,w,r') dw dr'}{1 - \iint r' \trho_t(q,w,r') dw dr'}\right]
      \trho_t(q,v,r) \right) = 0.
  \end{align}

\subsection{Fluctuations in the Euler regime}
\paragraph*{Mass and quasi particle fluctuations} The microscopic fluctuations of quasi-particles and lengths in the Euler regime jointly converge to a multitime Brownian field as shown by the following theorem.

\begin{theorem}[Fluctuations in the Euler regime]\label{tjd}
Let $\mu\in\cM$, and demote by $N^\ve$ the scaled empirical measure of a Poisson process $X^\ve$ with intensity $\mu$. Let $\etab$ be the multitime Brownian field (Lévy-Chentsov) with distance $\dd(a,b)=\mu_2(ab)$. The fluctuations of the mass and quasi particle positions jointly converge to non homogeneous multi time Brownian field, as follows
\begin{align}
    \ve^{-\frac12} \bigl(y_{N^{\ve}}(x,v;t)- y_\mu(x,v;t)\bigr)  &\tolaw \etab(x+vt,t), \label{jd1}\\
    \ve^{-\frac12} \bigl(\ve m_0^x(N^\ve_t)- m_0^x(\mu_t )\bigr)  &\tolaw \etab(x,t) -\etab(0,t). \label{jd2}
\end{align}
\end{theorem}

\begin{proof} Using \eqref{jf0}, \eqref{ey} and \eqref{e17} the left hand side of \eqref{jd1} is given by 
\begin{align}
  \label{ye2}
   \etab^\ve( x+vt,t)\tolaw \etab(x+vt,t),
\end{align}
by \eqref{e10}. 
Analogously, the left hand side of \eqref{jd2} is given by
\begin{align}
  \label{met}
 &\ve^{-\frac12} \bigl(H_{N^\ve}(x,t)-H_{\mu}(x,t) -( H_{N^\ve}(0,t)- H_{\mu}(0,t))\bigr)\\
                 &\qquad= \etab^\ve( x,t)-\etab^\ve(0,t)\tolaw \etab(x,t)-\etab(0,t). \qedhere
\end{align}
\end{proof}
We also have a functional central limit theorem for the one-dimensional marginals. 
\begin{corollary}[Time changed Brownian motion] Let $W$ be standard to side Brownian motion on $\RR$, pinned at the origin $W(0)=0$. The marginal quasi particle and mass fluctuations converge to time changed Brownian motion, 
  \begin{align}
   ( \hB^\ve(x,v;t))_{t\in\RR}     &\tolaw \bigl(W(\mu_2(b_0b_t))\bigr)_{t\in\RR}, & b_{t}&:=(x+vt,t)\\
    (\hB^\ve(x;t))_{t\in\RR}  &\tolaw  \bigl(W((\mu_t )_2(a_0a_x))\bigr)_{x\in\RR},&  a_x&:=(x,t)
  \end{align}
\end{corollary}
Recall that
  \begin{gather}
    \label{m2}
    \mu_2(b_0b_t) = \iiint r^2\,\Bigl(\one_{[x,\,x+(v-w)t]}(z)+ \one_{ [x+(v-w)t,\,x]}(z)\Bigr)\,d\mu(z,w,r),\\
    (\mu_t )_2(a_0a_x)=  \iiint r^2\,\Bigl(\one_{[-wt,x-wt]}(z) + \one_{[x-wt,-wt]}(z)\Bigr)\,d\mu(z,w,r).
  \end{gather}

\subsection{Fluctuations in the diffusive regime}

\emph{Quasi-particles in the diffusive regime. } 
We now consider time-order $\veo$, space-order~$1$, and intensity measure $\veo\mu$, so that $\vet$ lines cross segments with extremes at different times, and $\ve^{-1}$ lines cross segments with length order $1$. The fluctuations at time $t\neq0$ of the quasi particle starting at $x$ with velocity $v$ are given by
\begin{align}
  y_{N^\ve}( x,v;t\veo)- y_\mu( x,v;t\veo)
 &= j_{N^\ve}( x,v;t\veo)- j_\mu( x,v;t\veo) \label{jt8} \\
          &= \etab^\ve(x+vt\veo,t\veo)-\etab^\ve(x,0) \tolaw \etab(vt,t) 
            ,\label{jt7}
\end{align}
where $\etab^\ve$ is defined in \eqref{e17} and $\etab$ is the multitime Brownian field defined in \eqref{e10}. The process $(B(vt,t))_{t\in\RR}$ is one-dimensional Brownian motion indexed by $v$. There are $\ve^{-2}$ crossings, each contributing $\ve^2$ to $j_{N^\ve}$ at time $\ve^{-1}t$. After centering, the fluctuation in the number of crossings is of order $\ve^{-1}$, which, when multiplied by $\ve^2$, results in a quantity of order $\ve$. This corresponds to a diffusive scaling, where space is rescaled by a factor of  $\ve^{-1}$ and time by $\ve^{-2}$.

In this scale particles with the same velocity have the same fluctuations. In fact, the asymptotic covariance coincides with the variance of one of them,
\begin{align}
  \label{coo}
  \Cov(y_{N^\ve}( x,v;t\veo),y_{N^\ve}( \tx,v;t\veo)) \toe \mu_2(ob_{t}), \quad x,\tx\in \RR,  \quad b_t:=(vt,t).
\end{align}
Display (3.5) in our paper \cite{fo2024} is a particular case of \eqref{coo}.
If the velocities are distinct and $t\neq 0$, we have Levy-Chentsov covariances, and the limit is independent of the initial point,
\begin{align}
  \label{coh}
  \Cov(y_{N^\ve}\bigl( x,v;t\veo), y_{N^\ve}( x,\tv;t\veo)\bigr) &\toe \mu_2(ob_t\cap o\tb_t), \quad \tb_{t}:=(\tv t,t)\\
  &= \frac12\bigl(\mu_2(ob_t)+ \mu_2(o\tb_t)- \mu_2(b_t\tb_t)\bigr).
\end{align}
\begin{lemma}
  The following limit holds
  \begin{align}
   y_{N^{\ve}}(x,v;t\ve^{-1})-  j_\mu(x,v;t\ve^{-1}) \tolaw x+m_0^x(\mu) +  \etab(vt,t)\label{zo1} 
  \end{align}
\end{lemma}
\begin{proof} Using \eqref{y9}, the left hand side of \eqref{zo1} is equal to
\begin{align}
x+ m_0^x(N^\ve)+j_{N^{\ve}}(x,v;t\ve^{-1})-  j_\mu(x,v;t\ve^{-1}),
\end{align}
which converges in law to the right hand side of the same display, by \eqref{lmz} and \eqref{jt7}.
\end{proof}
Display (3.6) in our paper \cite{fo2024} is a particular case of \eqref{coh}, while 
convergence \eqref{zo1} corresponds to the limit (3.10) in our paper \cite{fo2024}. 
\vspace{4mm}

\bibliographystyle{plain}
\bibliography{bib-chentsov}
\end{document}